\documentclass{amsart}
       \usepackage{graphicx}
\usepackage{amssymb,amsthm,amsmath,epic,eepic,ulem}
\usepackage{alltt}

\newcommand{\tld}{\widetilde }
\input xy
\xyoption{all}
\newcommand{\DuBois}[1]{{\underline \Omega {}^0_{#1}}}
\newcommand{\FullDuBois}[1]{{\underline \Omega {}^{\mydot}_{#1}}}
\DeclareSymbolFont{rsfs}{OMS}{rsfs}{m}{n}
\DeclareSymbolFontAlphabet{\scr}{rsfs}
\newcommand{\sI}{\scr{I}}
\renewcommand{\O}{\mbox{$\mathcal{O}$}}
\newcommand{\mydot}{\centerdot}
\newcommand{\qis}{\simeq_{\text{qis}}}
\newcommand{\bQ}{\mathbb{Q}}
\newcommand{\ba}{\mathfrak{a}}
\newcommand{\mJ}{\mathcal{J}}
\newcommand{\bC}{\mathbb{C}}
\newcommand{\bH}{\mathbb{H}}
\newcommand{\sC}{\scr{C}}
\newcommand{\nsubset}{\not\subset}
\DeclareMathOperator{\Gr}{{Gr}}
\DeclareMathOperator{\an}{{an}}
\DeclareMathOperator{\Supp}{{Supp}}
\DeclareMathOperator{\coherent}{{coh}}
\DeclareMathOperator{\red}{red}
\DeclareMathOperator{\Spec}{{Spec}}

\newtheorem{theorem}{Theorem}[section]
\newtheorem{lemma}[theorem]{Lemma}
\newtheorem{proposition}[theorem]{Proposition}
\newtheorem{corollary}[theorem]{Corollary}

\theoremstyle{definition}
\newtheorem{definition}[theorem]{Definition}

\theoremstyle{remark}
\newtheorem{remark}[theorem]{Remark}
\newtheorem{question}[theorem]{Question}

\begin{document}

\title{A Simple Characterization of Du Bois Singularities}
\author{Karl Schwede}
\email{kschwede@umich.edu}
\begin{abstract}
We prove the following theorem characterizing Du Bois singularities.  Suppose that $Y$ is smooth and that $X$ is a reduced closed subscheme.  Let $\pi : \tld Y \rightarrow Y$ be a log resolution of $X$ in $Y$ that is an isomorphism outside of $X$.  If $E$ is the reduced pre-image of $X$ in $\tld Y$, then $X$ has Du Bois singularities if and only if the natural map $\O_X \rightarrow R \pi_* \O_E$ is a quasi-isomorphism.  We also deduce Koll\'ar's conjecture that log canonical singularities are Du Bois in the special case of a local complete intersection and prove other results related to adjunction.
\end{abstract}
\address{Department of Mathematics\\ University of Michigan\\ Ann Arbor, Michigan, 48109}
\maketitle

\section{Introduction and Background}
In this paper, we prove a simple new characterization of Du Bois singularities.  Inspired by this characterization, we deduce several new theorems, including results related to adjunction, and progress towards a conjecture of Koll\'ar.

Du Bois singularities were initially defined by Steenbrink as a setting where certain aspects of Hodge theory for smooth varieties still hold; see \cite{SteenbrinkCohomologicallyInsignificant}, \cite{SteenBrinkDuBoisReview} and \cite{DuBoisMain}.  They are defined by the cohomology of a complex which is difficult to understand since it requires the computation of resolutions of singularities for several varieties (specifically a simplicial or cubic hyperresolution is required; see  \cite{DeligneHodgeIII}, \cite{GNPP}, or \cite{CarlsonPolyhedral}).  Our new characterization of Du Bois singularities requires only a single resolution.  Our main result is:
\vskip 10pt
\hskip -16pt {\textsc{Theorem \ref{EasyDuBoisCriterion}.}}
{\it
Let $X$ be a reduced separated scheme of finite type over a field of characteristic zero.  Embed $X$ in a smooth scheme $Y$ and let $\pi : \tld Y \rightarrow Y$ be a log resolution of $X$ in $Y$ that is an isomorphism outside of $X$.  If $E$ is the reduced pre-image of $X$ in $\tld Y$, then $X$ has Du Bois singularities if and only if the natural map $\O_X \rightarrow R \pi_* \O_E$ is a quasi-isomorphism.
}\vskip 10pt
This result is actually a corollary of a stronger theorem (see Theorem \ref{AlternateDuBoisCharacterization}) in which we prove that $R \pi_* \O_E$ is quasi-isomorphic to $\DuBois{X}$, the complex used to determine whether or not $X$ has Du Bois singularities.
In addition to being much simpler, this new characterization unambiguously places Du Bois singularities among the pantheon of singularities of birational geometry.  The hypotheses of Theorem \ref{EasyDuBoisCriterion} can also be weakened in several ways; see Theorem \ref{AlternateDuBoisCharacterization} and Theorem \ref{DoubleLogResolution}.

Despite their complicated definition, Du Bois singularities were already known to be closely related to the singularities of birational geometry.  For example, it was conjectured by Koll\'ar that log canonical singularities are Du Bois; see \cite[1.13]{KollarFlipsAndAbundance}.  Conversely, it is known that normal quasi-Gorenstein Du Bois singularities are log canonical; see \cite[3.6]{KovacsDuBoisLC1}.  Furthermore, rational singularities are Du Bois, as shown in \cite{KovacsDuBoisLC1} (also see \cite{SteenbrinkMixed} for an alternate proof of the case of an isolated singularity).  In summary, Du Bois singularities fit naturally into the hierarchy of singularities in birational geometry as pictured in the following diagram.
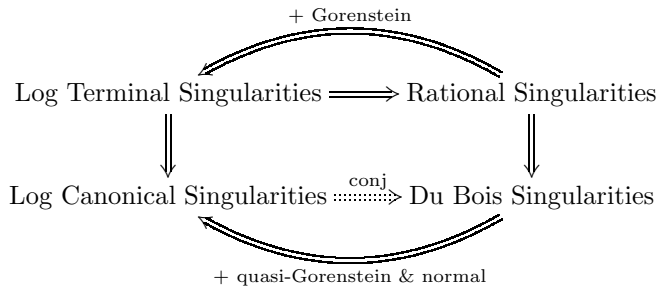
\[
\xymatrix{
\text{Log Terminal Singularities} \ar@{=>}[r] \ar@{=>}[d] & \text{Rational Singularities} \ar@{=>}[d] \ar_{\text{+ Gorenstein}}@/_2pc/@{=>}[l] \\
\text{Log Canonical Singularities} \ar^-{\text{conj}}@2{.>}[r] & \text{Du Bois Singularities} \ar^{\text{+ quasi-Gorenstein \& normal}}@/^2pc/@{=>}[l]\\
& \\
}
\]
We prove Koll\'ar's conjecture that log canonical singularities are Du Bois in the case of a local complete intersection; see Corollary \ref{LCHypersurfaceImpliesDuBois}.
Previous progress towards this conjecture was made by Ishii for isolated singularities (see \cite[2.4]{IshiiIsolatedQGorenstein}), and by Kov\'acs for Cohen-Macaulay singularities assuming the dimension of the singular locus is ``not too big'' (see \cite{KovacsDuBoisLC1} and \cite{KovacsDuBoisLC2}).

Because our new criterion for Du Bois singularities requires an embedding into an ambient scheme, it is natural to look for consequences related to adjunction.  The following result is inspired by our characterization.
\vskip 10pt
\hskip -16pt {\textsc{Theorem \ref{MultiplierIdealsAreDuBois}.}}
{\it
Suppose that $Y$ is a scheme with Kawamata log terminal singularities and that $\sI$ is an ideal sheaf on $Y$.  If the pair $(Y, \sI^r)$ is log canonical for some positive rational number $r$, then the multiplier ideal associated to the pair $(Y, \sI^r)$ cuts out a scheme with Du Bois singularities.
}
\vskip 10pt
Using similar methods, we also recover an alternate proof of a special case of Kawamata's subadjunction theorem \cite{KawamataSubadjunction2}; see Theorem \ref{MultiplierIdealsAreRational}.

We also prove a new theorem linking Du Bois and rational singularities, analogous to results of Koll\'ar and Shepherd-Barron \cite[5.1]{KollarShepherdBarron}, Karu \cite[2.5]{KaruBoundedness}, and Watanabe and Fedder \cite[2.13]{FedderWatanabe}.
\vskip 10pt
\hskip -16pt {\textsc{Theorem \ref{DuBoisImpliesRational}.}}
{\it
Suppose $X$ is a reduced scheme of finite type over a field of characteristic zero, $H$ is a Cartier divisor that has Du Bois singularities, and $X - H$ is smooth.  Then $X$ has rational singularities (in particular, it is Cohen-Macaulay).
}
\vskip 10pt
All schemes we consider will be separated and of finite type over a field of characteristic zero.  Typically they will also be reduced.  In particular, they may be non-normal and have several irreducible components of possibly different dimensions.  A smooth scheme will always be assumed to be equidimensional.

If $Y$ is a scheme, we will often work in the derived category of $\O_Y$-modules, denoted by $D(Y)$.  $D^b_{\coherent}(Y)$ will denote the derived category of bounded complexes of $\O_Y$ modules with coherent cohomology; see \cite{HartshorneResidues}.  In the setting of the derived category, we will write $F^{\mydot} \qis G^{\mydot}$ if $F^{\mydot}$ and $G^{\mydot}$ are quasi-isomorphic.

\vskip 12pt
\hskip -12pt
{\it Acknowledgements: }
Most of the results that appear in this paper originally appeared in my doctoral dissertation at the University of Washington which was directed by S\'andor Kov\'acs.  I would like to thank Howard Thompson for suggesting the use of marked ideals in Lemma \ref{HowardThompsonSuggestion}, Karen Smith and S\'andor Kov\'acs for their careful readings of this paper, and Mircea Musta\c{t}\v{a} and Shunsuke Takagi for several valuable discussions.

\section{Log Resolutions and Singularities of Pairs}

Let $Y$ and $\tld Y$ be reduced separated schemes of essentially finite type over a field $k$ of characteristic zero.  A \emph{birational morphism} $\pi : \tld Y \rightarrow Y$ is a morphism over $k$ that induces a bijection of generic points and induces a set of isomorphisms of residue fields at those points.

\begin{definition}
A \emph{resolution} of $Y$ is a proper birational morphism $\pi : \tld Y \rightarrow Y$ of schemes over $k$ such that $\tld Y$ is smooth over $k$.  A resolution is called a \emph{good resolution} if the exceptional set $E \subset \tld Y$ is a divisor with simple normal crossings.  If $\sI$ is an ideal sheaf on $Y$, a \emph{log resolution} of $\sI$ is a good resolution $\pi: \tld Y \rightarrow Y$ of $Y$ such that $\sI \O_{\tld Y} = \O_{\tld Y}(-G)$ is an invertible sheaf and such that $E \cup \Supp(G)$ is a simple normal crossings divisor.  Typically when performing a log resolution, $Y$ will be smooth or have some ``nice'' singularities.  A \emph{strong log resolution} of $\sI$ will be a log resolution $\pi : \tld Y \rightarrow Y$ of $\sI$ that is also an isomorphism outside the scheme $V(\sI)$ defined by $\sI$.
\end{definition}

\begin{remark}
Resolutions and log resolutions exist for schemes of finite type over a field of characteristic zero, \cite{HironakaResolution}.  A necessary and sufficient condition for the existence of strong log resolutions is that $Y - V(\sI)$ is smooth.
\end{remark}

We are now in a position to define rational singularities.

\begin{definition}
Suppose that $\pi : \tld Y \rightarrow Y$ is a resolution.  We say that $Y$ has \emph{rational singularities} if the natural map $\O_Y \rightarrow R \pi_* \O_{\tld Y}$ is a quasi-isomorphism.
\end{definition}

We now introduce the terminology necessary to define log terminal and log canonical singularities.

Let $Y$ be a $\bQ$-Gorenstein normal irreducible scheme of finite type over a field of characteristic zero, and let $X = \sum_{i=1}^k t_i X_i$ be a formal combination of closed subschemes $X_i \subsetneq Y$, defined by ideal sheaves $\ba_i \subset \O_X$ with $t_i \in \bQ$.  In our context we will always assume the $t_i$ are non-negative.  The notation $(Y, X)$ will be used to denote a pair of objects as above.  If $\ba$ is an ideal, we will also use $(Y, \ba^t)$ to denote $(Y, t V(\ba))$.  Sometimes, when $\ba$ is a reduced ideal and $X = V(\ba)$, we will use the notation $(Y, (r)X)$ to denote $(Y, V(\ba^{(r)}))$; that is, we take symbolic power instead of ordinary power.

Suppose $\pi : \tld Y \rightarrow Y$ is a birational morphism from a normal (typically smooth) scheme $\tld Y$ such that all ideal sheaves $\ba_i \O_{\tld Y} = \O_{\tld Y}(-G_i)$ are invertible.  Now we define the relative canonical divisor, so let $K_Y$ and $K_{\tld Y}$ denote canonical divisors of $Y$ and $\tld Y$.  The \emph{relative canonical divisor}, denoted by $K_{\tld Y / Y}$ is a $\bQ$-divisor supported on the exceptional locus of $\pi$.  To construct it explicitly, suppose that $n K_Y$ is Cartier.  Then we can pull back $n K_Y$ locally to obtain $\pi^*(n K_Y)$, a Cartier divisor on $\tld Y$.  The relative canonical divisor $K_{\tld Y / Y}$ is then defined as the unique divisor numerically equivalent to $K_{\tld Y} - (1/n) \pi^*(n K_Y)$ supported on the exceptional locus of $\pi$.  Finally we write
\[
K_{\tld Y/Y} - \sum_{i=1}^k t_i G_i = \sum_E a(E,Y,X) E
\]
where the $E$'s are arbitrary prime divisors on $\tld Y$.  Note that if the $t_i$ are integers, then $\ba_i^t \O_{\tld Y} = \O_{\tld Y}(-t_i G_i)$.  The rational number $a(E,Y,X)$ is called the \emph{discrepancy of the pair $(Y,X)$ along $E$}.

We now discuss singularities associated to pairs $(Y,X)$ where $X$ has arbitrary codimension in $Y$; see \cite{MustataSingularitiesOfPairs} and \cite{TakagiInversion}.  These definitions agree with the classical notions when $Y$ is smooth and $X$ is a divisor; see for example \cite{KollarMori}.

\begin{definition}
\label{SingularitiesInPairs}
We define a pair $(Y,X)$ to be \emph{log canonical (or lc)} if $a(E,Y,X) \geq -1$ for all divisors $E$ over $(Y,X)$ (where $E$ runs through all irreducible divisors of all birational morphisms such that the $\ba_i \O_{\tld Y}$ are invertible).  We define such a pair to be \emph{Kawamata log terminal (or klt)} if $a(E, Y, X) > -1$ for all $E$ as above.  We define such a pair to be \emph{purely log terminal (or plt)} if $a(E, Y, X) \geq -1$ for divisors $E$, dominating a component of $X$, and $a(E, Y, X) > -1$ for all other divisors.
\end{definition}

We will also use multiplier ideals and local vanishing for multiplier ideals; see \cite[9]{LazarsfeldPositivity2}.

\begin{definition}
\label{MultiplierIdealDefinition}
\numberwithin{equation}{theorem}
Let $Y$ be a $\bQ$-Gorenstein normal irreducible scheme of finite type over a field of characteristic zero.  Let $\sI$ be an ideal sheaf and $r$ a non-negative rational number.  Let $\pi : \tld Y \rightarrow Y$ be a log resolution of $\sI$ in $Y$ and let $G$ denote the divisor on $\tld Y$ such that $\O_{\tld Y}(-G) = \sI \O_{\tld Y}$.  We define the multiplier ideal of $\sI^r$, written as $\mJ(\sI^r)$, to be
\[
\pi_* \O_{\tld Y}(\lceil K_{\tld Y / Y} - r G \rceil).
\]
\end{definition}

Local vanishing for multiplier ideals roughly states that the higher direct images associated to the multiplier ideal are zero, or more specifically

\begin{theorem} [{\cite[9.4]{LazarsfeldPositivity2}}]
\label{MultiplierIdealVanishing}
Under the previous hypotheses, we have
\[
R^j \pi_* \O_{\tld Y}(\lceil K_{\tld Y /Y} - r G \rceil) = 0 \text{  for } j > 0.
\]
\end{theorem}

\section{Relevant Properties of Du Bois Singularities}
Let $X$ be a reduced and separated scheme over a field of characteristic zero.  Associated to $X$ is a filtered complex in the derived category of abelian sheaves on $X$, denoted by $(\FullDuBois{X}, F)$ which satisfies certain properties; see \cite{DuBoisMain}.  One of these properties is that the graded pieces associated to the filtration can be viewed as objects of $D^{b}_{\coherent}(X)$.  This filtered complex is unique up to filtered quasi-isomorphism which in particular implies the graded pieces are unique up to quasi-isomorphism in $D^{b}_{\coherent}(X)$.  One constructs this filtered complex by using a diagram of schemes called a (simplicial or cubic) hyperresolution $X_{\mydot} \rightarrow X$; see \cite{DuBoisMain}, \cite{GNPP}, \cite{SteenbrinkgVanishing}, and \cite{CarlsonPolyhedral}.  The fact that  $(\FullDuBois{X}, F)$ is independent of choice of hyperresolution is typically proved directly only over $\bC$, see for example \cite[V, 3.7]{GNPP}; however, the Lefschetz principle and certain base change theorems for hyperresolutions let one extend this result to any field of characteristic zero.  See \cite[III, 1.16]{GNPP} for additional discussion.

Let us sketch one construction of $(\FullDuBois{X}, F)$.  First, one considers a resolution $\pi$ of the scheme $X$ as well as something like resolutions of the exceptional set $E$ of $\pi$ and the image of $E$ in $X$.  These resolutions themselves have exceptional sets with images in the scheme we are resolving.  Both the exceptional sets and their images must also be resolved (in a particularly compatible way) and so on.  On each of these smooth schemes we constructed, we consider the usual De Rham complex with the ``filtration b\^ete''.  All of these De Rham complexes are then pushed down (in the derived category) to $X$  and combined in a right derived inverse limit, the result of which we call $(\FullDuBois{X}, F)$.  The filtration is induced by the filtrations on the smooth schemes; see \cite{GNPP}.  The complexity of the definition has been an obstacle to the study of Du Bois singularities.

Regardless, each graded piece $\Gr_F^i \FullDuBois{X}$ of $\FullDuBois{X}$ is an object of $D^b_{\coherent}(X)$ and can be thought of as a replacement for $\Omega_X^i$ in the case when $X$ is singular (in fact for $X$ smooth we have $\Gr_F^i \FullDuBois{X}[i] \qis \Omega^i_X$).  Therefore ${\underline \Omega {}^i_{X}}$ is often used to denote $\Gr_F^i \FullDuBois{X}[i]$.  If $X$ is projective over $\bC$, then there is a Hodge to De Rham-like spectral sequence, $E^{pq}_1 = \bH^q(X, \underline \Omega {}^p_{X} )$, which collapses at the $E_1$ stage, converges to $H^i(X^{\an}, \bC)$, and induces the usual Hodge structure of \cite{DeligneHodgeIII}; see \cite{DuBoisMain}.  In particular, for $X$ projective over $\bC$, one always has a surjection
\[
\xymatrix{H^i(X^{\an}, \bC) \ar@{->>}[r] & \bH^i(X, \DuBois{X})}.
\]   Because of this, one often works over $\bC$.

In this paper we are mainly concerned with properties of the zeroth graded piece of this complex, $\DuBois{X}$.  Thus we include a list of its relevant properties.  We will not make use of the Hodge to De Rham-like spectral sequence related to Du Bois singularities.

\begin{theorem}
\label{TheoremOldDuBoisProperties}
For every reduced separated scheme $Y$ over a field of characteristic zero, the complex $\DuBois{X} \in D^b_{\coherent}(X)$ has the following properties.
\numberwithin{equation}{theorem}
\begin{itemize}
\item[(i)]  \emph{Restriction to Open Sets: }  If $U \subset Y$ is open, then $\DuBois{Y}|_U \qis \DuBois{U}$.
\item[(ii)]  \emph{Functoriality:}  If $\phi : Y' \rightarrow Y$ is a morphism of finite type of reduced schemes, then it induces a natural map of objects in the derived category, $\DuBois{Y} \rightarrow R \phi_* \DuBois{Y'}$.
\item[(iii)]  \emph{Natural Maps:} There are natural maps
\[
\O_Y \rightarrow \DuBois{Y} \rightarrow R \pi_* \O_{\tld Y}
\]
in $D^b_{\coherent}(Y)$ where $\pi : \tld Y \rightarrow Y$ is a resolution of singularities of $Y$ and the composition is the usual map.  Both maps are quasi-isomorphisms if $Y$ is smooth.  The first map is a quasi-isomorphism if $Y$ has simple normal crossing singularities.
\item[(iv)]  \emph{An Exact Triangle:}  Suppose that $\pi : \tld Y \rightarrow Y$ is a projective morphism and $X \subset Y$ a reduced closed subscheme such that $\pi$ is an isomorphism outside of $X$.  Let $E$ denote the reduced subscheme of $\tld Y$ with support equal to $\pi^{-1}(X)$ and $\pi' : E \rightarrow X$ the induced map.  Then one has an exact triangle of objects in the derived category,
\begin{equation}
\label{fundamentalTriangle}
\xymatrix{
\DuBois{Y} \ar[r] & \DuBois{X} \oplus R \pi_* \DuBois{\tld Y} \ar[r]^-{-} & R \pi'_* \DuBois{E} \ar[r]^-{+1} & .\\
}
\end{equation}
In particular, if $\pi$ is a log resolution of $X$ in $Y$ (that is, $\tld Y$ is smooth and $E = (\pi^{-1}(X))_{\red}$ is a divisor with simple normal crossings), then we have the following triangle
\begin{equation}
\label{fundamentalResolutionTriangle}
\xymatrix{
\DuBois{Y} \ar[r] & \DuBois{X} \oplus R \pi_* \O_{\tld Y} \ar[r]^-{-} & R \pi_*' \O_E \ar[r]^-{+1} & .\\
}
\end{equation}
\end{itemize}
\end{theorem}
\begin{proof}
\begin{itemize}
\item[]
\item[(i)]  This follows from the construction of cubic hyperresolutions; see \cite{GNPP}.  Alternately see \cite[3.2]{DuBoisMain}.
\item[(ii)]  See \cite[3.21]{DuBoisMain} and \cite[I, 3.8 and V, 3.6]{GNPP}.
\item[(iii)]  See \cite[4.1]{DuBoisMain} and \cite[4.5]{DuBoisMain}.  The statements about isomorphisms for certain classes of singularities of $X$ follow directly from constructions of hyperresolutions.
\item[(iv)]  See \cite[4.11]{DuBoisMain}.  This is also particularly easy to see from the point of view of cubic hyperresolutions; see \cite{GNPP}.
\end{itemize}
\end{proof}

Now we define Du Bois singularities.

\begin{definition}
Suppose $X$ is a separated reduced scheme of finite type over a field of characteristic zero.  There is a natural map $\O_X \rightarrow \DuBois{X}$ by \ref{TheoremOldDuBoisProperties}(iii).  We say that $X$ has \emph{Du Bois singularities} if that map is a quasi-isomorphism.
\end{definition}

\begin{remark}
If $X$ is projective over $\bC$ and has Du Bois singularities, it is immediate from the Hodge to De Rham-like spectral sequence mentioned above that the natural map $H^i(X^{\an}, \bC) \rightarrow H^i(X, \O_X)$ is surjective.
\end{remark}

The following lemma about gluing Du Bois singularities can be used to construct numerous (non-normal) examples; see  \cite[3.8, 4.10]{DuBoisMain}.

\begin{lemma}
\label{GluingDuBois}
Suppose that $X$ is a reduced separated scheme of finite over a field of characteristic with closed subschemes $X_1$ and $X_2$ such that $X = X_1 \cup X_2$.  Suppose that $X_1$, $X_2$ and $X_1 \cap X_2$ all have Du Bois singularities (in particular we assume that $X_1 \cap X_2$ is reduced), then $X$ also has Du Bois singularities.
\end{lemma}
\begin{proof}
\cite[3.8, 4.10]{DuBoisMain} gives us an exact triangle
\[
\xymatrix{
\DuBois{X} \ar[r] & \DuBois{X_1} \oplus \DuBois{X_2} \ar[r] & \DuBois{X_1 \cap X_2} \ar[r]^-{+1} & .\\
}
\]
Note that the map
\[
\xymatrix{h^0(\DuBois{X_1}) \oplus h^0(\DuBois{X_2}) \ar[r] & h^0(\DuBois{X_1 \cap X_2})}
\]
is surjective.  Basic facts about pullback squares guarantee that $h^0(\DuBois{X})$ is isomorphic to $\O_X$, and the long exact sequence plus the surjectivity mentioned above guarantee that $h^i(\DuBois{X}) = 0$ for $i > 0$.  One can give a roughly equivalent proof using \ref{TheoremOldDuBoisProperties} part (iv) above instead of the results  \cite[3.8, 4.10]{DuBoisMain}.
\end{proof}

The following criterion of Kov\'acs is particularly useful for proving schemes have Du Bois singularities.

\begin{theorem}[{\cite[2.3]{KovacsDuBoisLC1}}]
\label{KovacsDuBoisSplitting}
Let $U$ be a reduced separated scheme of finite type over a field of characteristic zero such that $\O_U \rightarrow \DuBois{U}$ has a left inverse (that is a map $\DuBois{U} \rightarrow \O_U$ such that the composition $\O_U \rightarrow \DuBois{U} \rightarrow \O_U$ is a quasi-isomorphism); then $U$ has Du Bois singularities.
\end{theorem}

\begin{remark}
As an easy corollary, using \ref{TheoremOldDuBoisProperties} (iii), Kov\'acs proved that rational singularities are Du Bois.
\end{remark}

\begin{remark}
This was originally only proven over $\bC$.  The proof relies on the following fact:  If $X$ is projective over $\bC$ then there is a surjective map $H^i(X, \O_X) \rightarrow \bH^i(X, \DuBois{X})$.  This surjectivity is easily seen to be preserved by base change of our underlying field.  Using this fact, it is relatively easy to extend \ref{KovacsDuBoisSplitting} to other fields of characteristic zero.
\end{remark}

There is a corollary of this result that we will need; also see \cite[12.8]{KollarShafarevich} and \cite{KovacsRat}.

\begin{corollary}[{\cite[2.4]{KovacsDuBoisLC1}}]
\label{KovacsMapDuBoisSplitting}
Let $f : V \rightarrow U$ be a morphism of reduced separated schemes of finite type over a field of characteristic zero.  Suppose that $V$ has Du Bois singularities and that the natural map $\O_U \rightarrow Rf_* \O_V$ has a left inverse, then $U$ has Du Bois singularities as well.
\end{corollary}

\section{A Hyperresolution-Free Characterization of Du Bois Singularities}

We begin with a lemma that has been used implicitly in \cite[3.2]{KovacsDuBoisLC2} and \cite[7.7]{DuBoisDuality} in relation to Du Bois singularities. A precise statement is included since this result is critical in the construction of the new characterization of Du Bois singularities as well as certain applications.

\begin{lemma}
\label{TriangleSumToEqualCompletions}
Let $\sC$ be a triangulated category.  Suppose that we have objects $A, B, C, D$ and arrows $s : A \rightarrow B$, $t : A \rightarrow C$, $u : B \rightarrow D$, $v : C \rightarrow D$ such that
\[
\xymatrix{
A \ar[r]^-{s, -t} & B \oplus C \ar[r]^-{u + v} & D \ar[r]^-r & T(A) \\
}
\]
is an exact triangle.  If $\xymatrix{A \ar[r]^s & B \ar[r]^\phi & M \ar[r]^-{\psi} & T(A) }$ is an exact triangle, then there exists morphisms $f$ and $g$ such that
\[
\xymatrix{C \ar[r]^{v} & D \ar[r]^f & M \ar[r]^-g & T(C) }
\]
is an exact triangle.  Furthermore, we have a map of exact triangles
\[
\xymatrix{
A \ar[d]_t \ar[r]^s & B \ar[r]^\phi \ar[d]^u & M \ar[r]^{\psi} \ar@{=}[d]  & \\
C \ar[r]^{v} & D \ar[r]^f & M \ar[r]^{g} & .\\
}
\]
\end{lemma}
\begin{proof}
The proof of this lemma is a straightforward application of the octahedral axiom; see \cite[5.3.1]{KarlThesis} for details.
\end{proof}

We now apply the previous lemma in the case of the exact triangle, \ref{fundamentalTriangle}.

\begin{proposition}
\label{KeyTriangles}
Let $\pi : \tld X \rightarrow X$ be a proper birational morphism of reduced separated schemes over a field of characteristic zero such that $\tld X$ has Du Bois singularities.  Let $\Sigma$ be a closed subset of $X$ such that $\pi$ is an isomorphism outside of $\Sigma$.  Let $E$ be the reduced total transform of $\Sigma$, suppose $E$ also has Du Bois singularities, and let $\sI_E \subset \O_{\tld X}$ be the ideal sheaf of $E$ (these conditions are always satisfied when $\tld X$ is a strong log resolution of $\Sigma \subset X$).   Then the following is a map of exact triangles:
\[
\xymatrix{
R \pi_* \sI_E \ar[r] \ar@{=}[d] & \DuBois{X} \ar[r] \ar[d] & \DuBois{\Sigma} \ar[r]^-{+1} \ar[d] & \\
R \pi_* \sI_E \ar[r] & R \pi_* \O_{\tld X} \ar[r]^v & R \pi_* \O_{E} \ar[r]^-{+1} & \\
}
\]
\end{proposition}
\begin{proof}
We begin with the following exact triangle
\[
\xymatrix{
\DuBois{X} \ar[r] & R \pi_* \O_{\tld X} \oplus \DuBois{\Sigma} \ar[r]^-{v - u} & R \pi_* \O_{E} \ar[r]^-{+1} &
}
\]
which we obtain from \ref{fundamentalTriangle}.  We also have the exact triangle
\[
\xymatrix{
R \pi_* \sI_E \ar[r] & R \pi_* \O_{\tld X} \ar[r]^v & R \pi_* \O_{E} \ar[r]^-{+1} & \\
}
\]
simply arising from pushing down the corresponding short exact sequence.

Applying an easy sign switch followed by \ref{TriangleSumToEqualCompletions} gives us the desired result.
\end{proof}

We now can use the previous proposition to construct our new characterization.

\begin{theorem}
\label{AlternateDuBoisCharacterization}
Let $Y$ be a reduced separated scheme of finite type over a field $k$ of characteristic zero with rational singularities.  Let $X$ be a reduced subscheme of $Y$.  Let $\sI_X$ denote the ideal sheaf of $X$ in $Y$.
Assume there exists a proper birational map $\pi : \tld Y \rightarrow Y$ such that $\tld Y$ has rational singularities, the reduced pre-image of $X$ has Du Bois singularities and $\pi$ is an isomorphism outside of $X$.
If $E$ is the reduced total transform (pre-image) of $X$ then $\DuBois{X} \qis R \pi_* \O_E$.
\end{theorem}
\begin{proof}
Let $\sI_E$ be the ideal sheaf of $E \subset \tld Y$.  By \ref{KeyTriangles} we have a map of exact triangles
\[
\xymatrix{
R \pi_* \sI_E \ar[r] \ar@{=}[d] & \DuBois{Y} \ar[r] \ar[d] & \DuBois{X} \ar[d] \ar[r]^{+1} &\\
R \pi_* \sI_E \ar[r] & R \pi_* \O_{\tld Y} \ar[r] & R \pi_* \O_E \ar[r]^{+1} &\\
}
\]
Since $Y$ has rational singularities it has Du Bois singularities, so the middle vertical arrow is a quasi-isomorphism (with $\O_Y$).  The left vertical arrow is always a quasi-isomorphism, so the right vertical arrow is a quasi-isomorphism as well.  Since $\DuBois{X}$ is only defined up to quasi-isomorphism, this completes the proof.
\end{proof}

\begin{remark}
Note that the hypotheses of this theorem are satisfied if $Y$ is smooth and $\pi$ is a strong log resolution of $X$ in $Y$; that is, $\tld Y$ is smooth, the reduced pre-image of $X$ has simple normal crossings and $\pi$ is an isomorphism outside of $X$.
\end{remark}

The following special case of Theorem \ref{AlternateDuBoisCharacterization} is useful for comparing Du Bois and rational singularities.  It can also be used to easily reobtain the maps of \ref{TheoremOldDuBoisProperties} (iii) and the exact triangle \ref{fundamentalResolutionTriangle}; see \cite[5.3.7 and 5.3.9]{KarlThesis} for details.

\begin{corollary}
Suppose that $Y$ is smooth and $X$ is a reduced subscheme of $Y$.  Let $\pi : \tld Y \rightarrow Y$ be an embedded resolution of $X$ in $Y$ with simple normal crossings reduced exceptional divisor $E' \subset \tld Y$.  Let $\tld X$ denote the strict transform of $X$ and suppose that $E'$ meets $\tld X$ transversally in a simple normal crossings divisor on $\tld X$.  If $E = (\pi^{-1}(X))_{\red}$ then $R \pi_* \O_E \qis \DuBois{X}$.
\end{corollary}
\begin{proof}
Note that $E$ is the union of two reduced schemes with Du Bois singularities that intersect in a reduced scheme with Du Bois singularities, and thus it is Du Bois by \ref{GluingDuBois}.  Apply \ref{AlternateDuBoisCharacterization}.
\end{proof}

\begin{corollary}
\label{EasyDuBoisCriterion}
Let $X$ be a reduced separated scheme of finite type over a field of characteristic zero.  Suppose that $X \subseteq Y$ where $Y$ is smooth and suppose that $\pi : \tld Y \rightarrow Y$ is a log resolution of $X$ in $Y$ that is an isomorphism outside of $X$.  If $E$ is the reduced pre-image of $X$ in $\tld Y$, then $X$ has Du Bois singularities if and only if the natural map $\O_X \rightarrow R \pi_* \O_E$ is a quasi-isomorphism.
\end{corollary}

\begin{corollary}
With the above notation, $R \pi_* \O_E$ (as an object in $D^b_{\coherent}(X)$) is independent of the choice of embedding or resolution, up to quasi-isomorphism.
\end{corollary}

\begin{remark}
Using the above notation, it is easy to see that $\pi_* \O_E$ is the structure sheaf of the seminormalization (see \cite{GrecoTraversoSeminormal}) of $X$ when the field
being worked over is algebraically closed, \cite[5.4.17]{KarlThesis}.  This fact about Du Bois singularities and seminormality was also proven in \cite[5.2.2]{SaitoMixedHodge}.  Faithfully flat base change then allows us to note the following:
If $X$ is a reduced separated scheme of finite type over \emph{any} field of characteristic zero, then $X$ has Du Bois singularities if and only if $X$ is seminormal and $R^i \pi_* \O_E = 0$ for $i > 0$.  See \cite{SchwedeFInjectiveAreDuBois} for additional discussion of this issue.
\end{remark}

One limitation of this characterization is that it requires $\pi$ to be an isomorphism outside of $X$.  In applications, this can be a restrictive condition.  The following result allows us to weaken this hypothesis.

\begin{theorem}
\label{DoubleLogResolution}
Suppose that $Y$ is a smooth scheme of finite type over a field of characteristic zero, $X$ is a reduced closed subscheme of $Y$, and $\sI_X$ the ideal of $X$ in $Y$.  Suppose $\ba$ is any other sheaf of ideals on $Y$. Then there exists a simultaneous log resolution $\pi_2 : Y_2 \rightarrow Y$ of $\sI_X$ and $\ba$ such that if $E_2 = \pi_2^{-1}(X)$ is the reduced pre-image of $X$, then $R (\pi_2)_* \O_{E_2} \qis \DuBois{X}$.
\end{theorem}
\begin{proof}
Begin by taking a strong log resolution $\pi_1 : Y_1 \rightarrow Y$ of $\sI_X$ and consider $\ba_1 = \ba \O_{Y_1}$.  Let $E_1 = \pi_1^{-1}(X)$.  We claim that if we can construct a log resolution of $\ba_1$ by repeated blow-ups at smooth centers and such that the reduced pre-image of $E_1$  (which is the reduced pre-image of $X$) has simple normal crossings at each step; this would prove the proposition.  To see this, suppose $\pi_{1a} : Y_{1a} \rightarrow Y_1$ is the blow-up at a smooth center $C$ such that $E_{1a} = (\pi_{1a}^{-1}(E_1))_{\red}$ has simple normal crossings.  We claim that then $R (\pi_{1a})_* \O_{E_{1a}} \qis \O_{E_1}$.  There are two cases:
\begin{itemize}
\item[(1)]  $C \subset E_1$
\item[(2)]  $C \nsubset E_1$
\end{itemize}
In the first case, $R (\pi_{1a})_* \O_{E_{1a}} \qis \DuBois{E_1}$ by \ref{AlternateDuBoisCharacterization}, but simple normal crossing singularities have Du Bois singularities, so there is nothing to check.

In the second case, we need to check that $R (\pi_{1a})_* E_{1a} \qis \O_{E_1}$.  Because $Y_1$ has rational singularities, this is equivalent to showing that the natural map
\[
\O_{Y_1}(-E_1) \rightarrow R (\pi_{1a})_* \O_{Y_{1a}}(-E_{1a})
\]
is a quasi-isomorphism.  However, since the smooth center we blew-up was not contained in $E_1$, $\pi_{1a}^*(E_1) = E_{1a}$.  Thus we see that
\[
\O_{Y_1}(-E_1) \rightarrow R (\pi_{1a})_* (\O_{Y_{1a}}(-E_{1a}))
\]
is an isomorphism by an application of the projection formula (again we use that $Y_1$ has rational singularities).

By repeating this argument, we see that if $\pi_2 : Y_2 \rightarrow Y_1$ is a log resolution of $\ba_1$ obtained in this way, $R (\pi_2)_* \O_{E_2} \qis \O_{E_1}$.

To show that such a sequence of blow-ups exists, we need the following lemma which is actually an easy corollary of an algorithmic desingularization theorem; see for example \cite{WlodarczykResolution}, \cite{BravoEncinasVillmayorSimplified}.

\begin{lemma}
\label{HowardThompsonSuggestion}
Suppose $Y_1$ is a smooth scheme of finite type over a field of characteristic zero, $\ba_1$ is an ideal sheaf on $Y_1$ and $E_1$ a simple normal crossings divisor on $Y_1$.  Then there exists a log resolution $\pi : Y_2 \rightarrow Y_1$ obtained by successive blow-ups of smooth centers that have simple normal crossings with the pre-image of $E_1$ at \emph{every} step.
\end{lemma}
I would like to thank Howard Thompson for suggesting the ``hijacking'' of the resolution of marked ideals/basic objects used to prove this lemma.
\begin{proof}
In \cite{WlodarczykResolution} (respectively \cite{BravoEncinasVillmayorSimplified}), instead of simply resolving ideals, for recursive purposes one resolves a tuple $(Y_1, \ba_1, E_1, 1)$ called a marked ideal (respectively $(Y_1, (\ba_1, 1), E_1)$ called a basic object).  A \emph{resolution} of one of these objects is a chain of blow-ups at smooth centers such that the centers have simple normal crossings with the $E_i$ object, and furthermore, the $E_i$-term only changes by adding the new exceptional set at each stage; see \cite[2.13]{WlodarczykResolution} (respectively \cite[2.3, 3.6]{BravoEncinasVillmayorSimplified}).  \emph{Resolutions} of marked ideals (respectively, basic objects) exist by \cite[4.0.1]{WlodarczykResolution} (respectively, \cite[3.10]{BravoEncinasVillmayorSimplified}), which completes the proof of the lemma.
\end{proof}

Now we return to the proof of \ref{DoubleLogResolution}.  Since the centers at each stage have simple normal crossings with the pre-image of $E_1$, we see that this lemma completes the proof of the proposition.
\end{proof}

Additional generalizations that do not require the precise control of the resolution would be desirable; see section \ref{Questions}.

\section{Applications}

In this section we apply our results to other problems in birational geometry.

The following theorem was inspired by an analogous result of Fedder and Watanabe involving $F$-injective and $F$-rational singularities; see \cite[2.13]{FedderWatanabe}.  A similar statement involving semi-log canonical and canonical singularities can be found in \cite[5.1]{KollarShepherdBarron} and a further generalization can be found in \cite[2.5]{KaruBoundedness}.

\begin{theorem}
\label{DuBoisImpliesRational}
Suppose $X$ is a reduced scheme of finite type over a field of characteristic zero, $H$ is a Cartier divisor that has Du Bois singularities, and $X - H$ is smooth.  Then $X$ has rational singularities (in particular, it is Cohen-Macaulay).
\end{theorem}

Our proof of this theorem relies on the following result of Kov\'acs.

\begin{theorem}[{\cite{KovacsRat}}]
\label{KovacsRationalCriterion}
Suppose that $X$ and $\tld X$ are reduced irreducible separated schemes of finite type over a field of characteristic zero $k$.  Further suppose that $\pi : \tld X \rightarrow X$ is a map over $k$ and that $\tld X$ has rational singularities.  If the natural map $\O_X \rightarrow R \pi_* \O_{\tld X}$ has a left inverse (that is, there exists a map $R \pi_* \O_{\tld X} \rightarrow \O_X$ such that the composition, $\O_X \rightarrow R \pi_* \O_{\tld X} \rightarrow \O_X$ is a quasi-isomorphism), then $X$ has rational singularities as well.
\end{theorem}
The proof of Theorem \ref{KovacsRationalCriterion} relies on the existence of resolutions of singularities (see \cite{HironakaResolution}), Grothendieck duality (see \cite{HartshorneResidues}), and Grauert-Riemenschneider vanishing (see \cite{GRVanishing}).  Note that $\pi$ is not required to be birational.  Compare with \cite[2.3]{KovacsDuBoisLC1} which is stated previously in this paper in \ref{KovacsDuBoisSplitting}.

\begin{remark}
\label{KovacsRationalIrreducibleRemark}
The hypothesis that $X$ and $\tld X$ are irreducible can be removed if we require that every irreducible component of $\tld X$ dominates an irreducible component of $X$.
\end{remark}

We now prove \ref{DuBoisImpliesRational}.

\begin{proof}
Without loss of generality we may assume that $X = \Spec S$ is affine and that $H = V(f)$ where $f \in S$ is a regular element.
Using this notation, since $\Spec S[f^{-1}]$ is smooth and $f$ is a regular element, it is easy to see that $X$ is smooth in codimension 1.  Let $\pi : \tld X \rightarrow X$ be a strong log resolution of $(X, H)$ such that the strict transform $\tld H$ of $H$ is also smooth (and thus in particular, $\pi|_{\tld H}$ is a resolution of $H$).  We use  $\overline H$ to denote the reduced total transform of $H$ and we observe that it is a divisor with simple normal crossings.  Therefore $\overline H$ has Du Bois singularities.
We obtain the exact triangle
\[
\xymatrix{
R \pi_* \O_{\tld X}(-\overline H) \ar[r] & \DuBois{X} \ar[r] & \DuBois{H} \ar[r]^-{+1} & .
}
\]
from Proposition \ref{KeyTriangles}.
Now consider the following diagram
\[
\xymatrix{
f \O_X \ar[d] \\
R \pi_* f \O_{\tld X} \ar[d] \\
R \pi_* \O_{\tld X}(-\overline H) \ar[r] \ar[d] & \DuBois{X} \ar[r] \ar[dl] & \DuBois{H} \ar[r]^{+1} & \\
R \pi_* \O_{\tld X} \\
}
\]
Here we think of $f \O_{\tld X}$ as an ideal sheaf (and one that is certainly contained in $\O_{\tld X}(-\overline H)$ since the reduced divisor corresponding to $f \O_{\tld X}$ is just $-\overline H$).
Since $H$ is Du Bois, $\DuBois{H} \qis \O_H$.  But since $X$ is Du Bois outside of $H$, $X$ is itself Du Bois \cite[3.2]{KovacsDuBoisLC2}, and in particular $\DuBois{X} \qis \O_X$.  Therefore $R \pi_* \O_{\tld X}(-\overline H) \qis f \O_X$.  Finally, consider the image of $f \O_X$ in $R \pi_* \O_{\tld X}$.  Since $R \pi_* \O_{\tld X}(-\overline H) \rightarrow R \pi_* \O_{\tld X}$ factors through $\DuBois{X}$, we see the image of $R \pi_* \O_{\tld X}(-\overline H)$ also agrees with that of $f \O_X$ in $R \pi_* \O_{\tld X}$.
This gives us the following composition
\[
\xymatrix{
f \O_X \ar[r] & R \pi_* (f \O_{\tld X}) \ar[r] & f \O_X
}
\]
which is a quasi-isomorphism.  Finally, abstractly (but compatibly), $f \O_X \cong \O_X$ and $f \O_{\tld X} \cong \O_{\tld X}$, which gives us the desired result by \ref{KovacsRationalCriterion}.
\end{proof}

\begin{remark}
If \ref{LogRationalificationsQuestion} holds (or if certain natural generalizations to \ref{QuestionLogResolutionGeneralization} hold), then one could replace the condition that $X-H$ is smooth with the condition that $X-H$ has rational singularities.
\end{remark}

Our alternate characterization of Du Bois singularities \ref{AlternateDuBoisCharacterization} and related results in \cite{TakagiInversion} (see \ref{LCPairImpliesDuBois}) inspired the following adjunction-like theorem for log canonical and Du Bois singularities; also compare with \cite{VassilevTestIdeals} and \cite{AmbroSeminormalLocus}.

\begin{theorem}
\label{MultiplierIdealsAreDuBois}
Suppose that $Y$ is a scheme of finite type over a field of characteristic zero with Kawamata log terminal singularities, and that $\sI$ is an ideal sheaf on $Y$.  If the pair $(Y, \sI^r)$ is log canonical for some positive rational number $r$, then $\mJ(\sI^r)$ cuts out a scheme with Du Bois singularities.
\end{theorem}
\begin{proof}
Without loss of generality, we may assume that $Y$ is affine.  Let $X = V(\mJ(\sI^r))$.  Let $\pi : \tld Y \rightarrow Y$ be a simultaneous log resolution of $\sI$ and $\mJ(\sI^r)$. In particular, $\tld Y$ is smooth and the reduced pre-images of $\mJ(\sI^r)$ and $\sI$ are divisors with simple normal crossings.  Let us use $G$ to denote the divisor on $\tld Y$ such that $\sI \O_{\tld Y} = \O_{\tld Y}(-G)$ and let us use $E$ to denote $(\pi^{-1}(X))_{\red}$.  Note that $E$ has Du Bois singularities.

Since the pair is log canonical and all the divisors with discrepancy equal to $-1$ are centered over $X$, we have $\lceil K_{\tld Y / Y} - rG \rceil \geq -E$ and therefore we have an inclusion $\O_{\tld Y}(-E) \subset \O_{\tld Y}(\lceil K_{\tld Y /Y} -rG \rceil)$.  Thus we obtain the composition
\[
\xymatrix@=12pt{
\mJ(\sI^r) \ar[r] & R \pi_*( \mJ(\sI^r) \O_{\tld Y}) \ar[r] & R \pi_* \O_{\tld Y}(-E) \ar[r] & R \pi_*\O_{\tld Y}(\lceil K_{\tld Y / Y} -rG \rceil) \qis \mJ(\sI^r) \\
}
\]
where the last isomorphism is the local vanishing theorem for multiplier ideals, \cite[9.4.4]{LazarsfeldPositivity2}.
The total composition is a quasi-isomorphism and can be thought of as a splitting (technically a left inverse) of $\xymatrix{\mJ(\sI^r)\ar[r] & R \pi_* \O_{\tld Y}(-E)}$ in the derived category.  Note we can fit the composition into a larger diagram in the derived category,
\[
\xymatrix{
 \mJ(\sI^r) \ar[r] \ar[r] \ar[d] & \O_Y \ar[r] \ar@{=}[d] & \O_X \ar[d] \ar[r]^{+1} & \\
 R \pi_* \O_{\tld Y}(-E) \ar[r] \ar[d] & R \pi_* \O_{\tld Y} \ar[r] \ar@{=}[d] & R \pi_* \O_E \ar@{.>}[dd] \ar[r]^-{+1} & \\
 R \pi_* \O_{\tld Y}(\lceil K_{\tld Y / Y} - rG \rceil) \ar[r] \ar@{=}[d] & R \pi_* \O_{\tld Y}(\lceil K_{\tld Y / Y}\rceil) \ar@{=}[d] & \\
 \mJ(\sI^r) \ar[r] & \O_Y \ar[r] & \O_X \ar[r]^{+1} & \\
}
\]
where the dotted arrow exists because the derived category is a triangulated category.  However, since the left two total vertical compositions are quasi-isomorphisms, the right total vertical composition is also a quasi-isomorphism, which proves that $X$ is Du Bois by \cite[2.4]{KovacsDuBoisLC1}, or as stated earlier in \ref{KovacsMapDuBoisSplitting}.
\end{proof}

\begin{remark}
We note that if $Y$ is smooth, then we can choose $\pi$ so that the object $R \pi_* \O_E$ in the proof above is quasi-isomorphic to $\DuBois{X}$ by \ref{DoubleLogResolution}.  We expect this is always true, see \ref{QuestionLogResolutionGeneralization}.
\end{remark}

We may use this theorem to prove a corollary related to a conjecture of Takagi, \cite[4.4]{TakagiInversion}.  Also see \cite{KawakitaComparisonNonLCI}, a recent paper by Kawakita.  First, however, we need a lemma about multiplier ideals.

\begin{lemma}
\label{MultiplierIdealLemma}
Suppose that $Y$ is a smooth scheme of finite type over a field of characteristic zero and that $X \subset Y$ is a reduced closed subscheme of pure codimension $r$ with ideal sheaf $\sI_X$.  If the pair $(Y, rX)$ (or $(Y, (r)X)$) is log canonical (see \ref{SingularitiesInPairs}), then the multiplier ideal $\mJ(\sI_X^r)$ (or $\mJ(\sI_X^{(r)})$) is equal to $\sI_X$.
\end{lemma}
\begin{proof}
The associated multiplier ideal is reduced since the pair is log canonical, it has the right support since the multiplicity of $rX$ (respectively $(r)X$) along the generic points of $X$, is equal to the codimension.
\end{proof}

\begin{corollary}
\label{LCPairImpliesDuBois}
Let $(Y, X)$ be a pair where $Y$ is a smooth scheme of finite type over a field of characteristic zero and $X$ is a reduced subscheme of pure codimension $r$.  If the pair $(Y, rX)$ (or $(Y, (r)X)$) is log canonical, then $X$ has Du Bois singularities.
\end{corollary}

It is not difficult to construct examples that show that the converse of this corollary is false in general.  Consider the subscheme defined by $I = (uv, uz, z(v-y^2))$ inside $\Spec k[u,v,y,z]$.  One can verify that $(\Spec k[u,v,y,z], V(I^{(2)}))$ is not log canonical by performing a log resolution and it is easy to see that $V(I)$ has Du Bois singularities using \ref{GluingDuBois}.  Note this example also occurs in a related situation in \cite[3.2]{SinghDeformationOfFPurity}, also see \cite{FedderFPureRat}.

\begin{remark}
A similar statement to \ref{LCPairImpliesDuBois} can be made in the non-equidimensional case assuming the ideal vanishes to the order of the codimension along each irreducible component.
\end{remark}

We can also use the corollary to prove that log canonical local complete intersections are Du Bois.

\begin{corollary}
\label{LCHypersurfaceImpliesDuBois}
Suppose that $X$ is a normal local complete intersection.  Then $X$ has Du Bois singularities if and only if it has log canonical singularities.
\end{corollary}
\begin{proof}
First assume that $X$ has log canonical singularities.  The statement is local, so we assume $X$ is affine and that $X$ embeds in a smooth scheme $Y$ as a complete intersection (note in particular it is Gorenstein).  The condition that $X$ is log canonical is then equivalent to the condition that $(Y, rX)$ is log canonical by \cite[3.2]{EinMustata}.  Apply \ref{LCPairImpliesDuBois} to obtain one direction of the equivalence.  The converse is just \cite[3.6]{KovacsDuBoisLC1}.
\end{proof}

A technique similar to that used in \ref{MultiplierIdealsAreDuBois} can also be used to prove a special case of Kawamata's subadjunction theorem; see \cite{KawamataSubadjunction2}.  Instead of using the ``splitting'' for Du Bois singularities used above, we use the analogous theorem for rational singularities stated previously in \ref{KovacsRationalCriterion}.

\begin{theorem}\cite{KawamataSubadjunction2}
\label{MultiplierIdealsAreRational}
Suppose that $Y$ is Kawamata log terminal, $\sI$ is an ideal sheaf, $r$ is a positive rational number, and the pair $(Y, \sI^r)$ is log canonical.  Further suppose that the pair has a log resolution $\pi : \tld Y \rightarrow Y$ that \emph{only} achieves a discrepancy of $-1$ along a set of divisors dominating irreducible components of $V(\mJ(\sI^r))$ and such that these $-1$ divisors are disjoint in $\tld Y$.  Then $X = V(\mJ(\sI^r))$ has rational singularities.
\end{theorem}
\begin{proof}
Let $E'$ denote the divisor on $\tld Y$ made up of components with a discrepancy of -1.  Let us use $G$ to denote the divisor on $\tld Y$ such that $\sI^r \O_{\tld Y} = \O_{\tld Y}(-G)$.  The hypotheses imply that $E'$ is smooth and that $\O_{\tld Y}(-E') \subset \O_{\tld Y}(\lceil K_{\tld Y / Y} - rG\rceil)$.  We thus obtain the following composition
\[
\xymatrix@=12pt{
\mJ(\sI^r) \ar[r] & R \pi_* \mJ(\sI^r) \O_{\tld Y} \ar[r] & R \pi_* \O_{\tld Y}(-E') \ar[r] & R \pi_*( \O_{\tld Y}(\lceil K_{\tld Y / Y} -rG \rceil)) \qis \mJ(\sI^r) \\
}
\]
as before.  We now mirror the argument found in \ref{MultiplierIdealsAreDuBois} by fitting this ``splitting'' of the map $\mJ(\sI^r) \rightarrow R \pi_* \O_{\tld Y}(-E')$ into a larger diagram,
\[
\xymatrix{
 \mJ(\sI^r) \ar[r] \ar[r] \ar[d] & \O_Y \ar[r] \ar@{=}[d] & \O_X \ar[d] \ar[r]^{+1} & \\
 R \pi_* \O_{\tld Y}(-E') \ar[r] \ar[d] & R \pi_* \O_{\tld Y} \ar[r] \ar@{=}[d] & R \pi_* \O_{E'} \ar@{.>}[dd] \ar[r]^-{+1} & \\
 R \pi_* \O_{\tld Y}(\lceil K_{\tld Y / Y} - rG\rceil) \ar[r] \ar@{=}[d] & R \pi_* \O_{\tld Y}(\lceil K_{\tld Y / Y} \rceil) \ar@{=}[d] & \\
 \mJ(\sI^r) \ar[r] & \O_Y \ar[r] & \O_X \ar[r]^{+1} & .\\
}
\]
We then apply \ref{KovacsRationalCriterion} and \ref{KovacsRationalIrreducibleRemark} to complete the proof.
\end{proof}

\begin{corollary}
\label{pltImpliesRat}
Suppose $(Y, X)$ is a pair where $Y$ is a smooth scheme of finite type over a field of characteristic zero and $X$ is a reduced subscheme of pure codimension $r$.  If the pair $(Y, rX)$ (or $(Y, (r)X)$) is purely log terminal (see \ref{SingularitiesInPairs}), then $X$ has rational singularities.
\end{corollary}
\begin{proof}
Let $\sI_X$ be the ideal sheaf of $X$ and let $\sI$ be the ideal sheaf $\sI_X^r$ of $rX$ (or $\sI_X^{(r)}$ of $(r)X$) in $Y$.  We note that as before, $\mJ(\sI) = \sI_X$.  Thus the only issue is proving that such a pair satisfies the conditions of \ref{MultiplierIdealsAreRational}.

We let $\pi : \tld Y \rightarrow Y$ be a log resolution of $\sI_X$ (that is simultaneously a log resolution of $\sI_X^{(r)}$ if applicable).  Let $E$ be the reduced exceptional divisor of $\pi$.  Let $E'$ be the sub-divisor of $E$ made up of components of $E$ that dominate components of $X$.  We claim that we may assume that $E'$ is in fact smooth and each divisor of $E'$ is obtained from blowing up the generic point of an irreducible component of $X$.  To see this claim, first perform a log resolution $Y_1 \rightarrow Y$ of $\sI_X$ with this property; see \cite[7.3.(iii)]{BravoEncinasVillmayorSimplified}.  This is a log resolution of $\sI$ generically.  Therefore, if we now follow a modern algorithm for a log resolution of $\sI \O_{Y_1}$ (for example \cite{BravoEncinasVillmayorSimplified}, or \cite{WlodarczykResolution}), it is easy to see that all future blow-ups will be centered over a lower codimension subscheme of $X$.

In summary, we may assume that each component of $X$ has exactly one component of $E$ dominating it and these components are disjoint.  At this point the definition of purely log terminal allows us to apply Theorem \ref{MultiplierIdealsAreRational}, which completes the proof.
\end{proof}

\section{Further Questions}
\label{Questions}

Let $Y$ be an ambient space with rational singularities and let $X$ be a subscheme that we are trying to determine whether or not $\O_X \qis \DuBois{X}$.  The main limitation of the characterizations of Du Bois singularities contained in this paper is that, in most cases, we cannot modify $Y$ outside of $X$.  In many applications this is not optimal.  Therefore we have the following question:

\begin{question}
\label{QuestionLogResolutionGeneralization}
Suppose that $Y$ has rational singularities and $X \subset Y$.  Let $\pi : \tld Y \rightarrow Y$ be \emph{any} log resolution of the pair $(Y, X)$.  Let $E$ denote the reduced scheme with support $\pi^{-1}(X)$.  Is it true that $R \pi_* \O_E \qis \DuBois{X}$?
\end{question}

There is another direction one can explore.  Is there a way to abstract the properties of $E$ which determine $\DuBois{X}$ without the need for an embedding.  Specifically,

\begin{question}
Given a reduced scheme $X$ of finite type over a field of characteristic zero, can one specify properties of a map $\pi : E
\rightarrow X$ (without reference to an embedding of $X$) which guarantee that $R \pi_* \O_E {\qis} \DuBois{X}$? And, given such
a definition, can one show that reasonable additional requirements (such as whether any two such maps can be dominated by a
third) are met?
\end{question}

Some possible conditions on $\pi$ include the following.
For each $x \in X$, let $\pi^{-1}(x)$ denote the reduced fiber.
\begin{itemize}
\item[(i)]  $\pi$ should be a proper surjective map (with connected fibers).
\item[(ii)]  $k(x) \qis R (\pi |_{\pi^{-1}(x)})_* \O_{\pi^{-1}(x)}$.
\item[(iii)]  Some sort of requirement on the singularities of $\pi^{-1}(x)$ or perhaps a requirement on the singularities of $E$ (simple normal crossings might be reasonable).
\end{itemize}

The following related question, which can be thought of as a strengthening of a conjecture of Koll\'ar, \cite[Chapter 12]{KollarShafarevich}, would also be useful in certain applications.

\begin{question}
\label{LogRationalificationsQuestion}
Suppose $Y$ is a reduced separated scheme of finite type over a field of characteristic zero, $X \subset Y$ is a reduced closed subscheme, and $Y-X$ has rational singularities.  Does there exist a proper birational map $\pi$ from a scheme $\tld Y$ to $Y$ such that
\begin{itemize}
\item[(i)]  $\tld Y$ has rational singularities,
\item[(ii)]  $\pi$ is an isomorphism outside of $X$, and
\item[(iii)]  the reduced scheme corresponding to $\pi^{-1}(X)$ has Du Bois singularities?
\end{itemize}
\end{question}

It is also reasonable to ask whether the consequence of the adjunction-type theorem of \ref{MultiplierIdealsAreDuBois} can be strengthened; see \cite{VassilevTestIdeals} and \cite{TakagiInversion}.  Specifically,

\begin{question}
Is there a generalization of semi-log-canonical singularities (see \cite{KollarFlipsAndAbundance}) such that if $Y$ is Kawamata log terminal, $\sI$ an ideal sheaf, $r$ a positive rational number and $(Y, \sI^r)$ is log canonical, then is $V(\mJ(\sI^r))$ semi-log canonical?
\end{question}

A generalization of semi-log-canonical is truly needed here.  Even beyond the $\bQ$-Gorenstein hypothesis, semi-log canonical singularities are typically assumed to be S2, have simple normal double crossings in codimension one, and be equidimensional.  One can construct examples where the pair $(Y, \sI^r)$ is log canonical but $V(\mJ(\sI^r))$ is neither S2, nor has simple normal double crossings in codimension one, nor is equidimensional.

Finally, we should mention the possible idea of something like a converse to \ref{LCPairImpliesDuBois}.

\begin{question}
Suppose that $Y$ is smooth and $X \subset Y$ is log canonical (or an appropriate generalization of semi-log canonical) of pure codimension $r$.  Is it true that the pair $(Y, (r)X)$ is log canonical?
\end{question}

An affirmative answer to this question would imply that log canonical singularities are Du Bois by \ref{LCPairImpliesDuBois}, and likely have other applications as well.





\providecommand{\bysame}{\leavevmode\hbox to3em{\hrulefill}\thinspace}
\providecommand{\MR}{\relax\ifhmode\unskip\space\fi MR}
\providecommand{\MRhref}[2]{%
  \href{http://www.ams.org/mathscinet-getitem?mr=#1}{#2}
}
\providecommand{\href}[2]{#2}

\end{document}